\documentclass[11pt,leqno]{article}
\usepackage{geometry}
\newgeometry{vmargin={28mm}, hmargin={35mm,35mm}}   

\usepackage{xcolor}

\usepackage[hidelinks]{hyperref}

\usepackage{amsmath,amsthm,amsfonts,amssymb,latexsym,amscd,enumerate}
\usepackage{slashed}

\usepackage{palatino}
\usepackage[mathcal]{euler}

\usepackage{xy}
\xyoption{all}

\numberwithin{equation}{section}
\swapnumbers

\newtheorem{theorem}{Theorem}[section]
\newtheorem*{theorem*}{Theorem}

\newtheorem{proposition}[theorem]{Proposition}
\newtheorem*{proposition*}{Proposition}
\newtheorem{lemma}[theorem]{Lemma}

\theoremstyle{definition}
\newtheorem{definition}[theorem]{Definition}

\newtheorem{example}[theorem]{Example}

\newtheorem{remark}[theorem]{Remark}

\DeclareMathOperator{\End}{End}

\begin{document}

	\title{A Groupoid Proof of The Lefschetz fixed point formula}
	
	\author{Zelin Yi}

	\date{}

	\maketitle
	
	\begin{abstract}
	The purpose of this article is to present a "Groupoid proof" to the Lefschetz fixed point formula for elliptic complexes. We shall define a "relative version" of tangent groupoid, describe the corresponding pseudodifferential calculi and explain the relation with the Lefschetz fixed point formula.
	\end{abstract}
	
	\section{Introduction}
	 The notion of tangent groupoids is invented by Alain Connes\cite{Connes94} to simplify the proof of the Atiyah Singer index theorem(see also \cite{Higson10}). In \cite{HigsonYi19}, the authors construct a "rescaled vector bundle" over the tangent groupoid whose space of smooth sections supports a continuous family of supertraces with Getzler rescaling built in. The existence of such a continuous family of supertraces can be thought of as a version of the index theorem. In contrast, the Lefschetz fixed point formula requires a "rescaling" in a differential direction. In this paper,  instead of building a "rescaled bundle", we shall modify the tangent groupoid construction to encode the appropriate "Lefschetz type rescaling". 
	 
	 This paper grow out of an effort to apply \cite{HigsonYi19} to Bismut's hypoelliptic Laplacian \cite{Bismut11} where both Getzler rescaling and "Lefschetz type rescaling" are needed, the detail of which we plan to pursue elsewhere.
	
	To fix the notations, let us quickly go through the basic setting of the Lefschetz fixed point formula\cite{AtiyahBott66}. Let $V$ be a compact manifold, $E_0,E_1,\cdots ,E_N$ be a sequence of vector bundles over $V$. Let 
	\begin{equation}\label{elliptic-complex}
	0\to C^\infty(V,E_0) \xrightarrow{d_0} C^\infty(V,E_1) \xrightarrow{d_1} \cdots \xrightarrow{d_{N-1}} C^\infty(V,E_N) \to 0
	\end{equation}
	be a sequence of differential operators such that $d_{i+1}d_i=0$. By standard terminology, such a sequence is called a \emph{complex}. An \emph{Elliptic complex} over $V$ is a complex such that its sequence of principal symbols
	\[
	0\to C^\infty(T^\ast V,\pi^\ast E_0) \xrightarrow{\sigma(d_0)} C^\infty(T^\ast V,\pi^\ast E_1) \xrightarrow{\sigma(d_1)} \cdots \xrightarrow{\sigma(d_{N-1})} C^\infty(T^\ast V,\pi^\ast E_N) \to 0
	\]
	is exact outside the zero section of $T^\ast V$. Let 
	\begin{equation}\label{eq-vector-bundle-E}
	E=\bigoplus_i E_i
	\end{equation}
	be the direct sum vector bundle and 
	\begin{equation}\label{eq-sum-diff}
	d=\bigoplus_i d_i: C^\infty(V,E)\to C^\infty(V,E)
	\end{equation}
	be the corresponding direct sum of differential operators. It is convenient, for our purpose, to consider the $\mathbb{Z}/2\mathbb{Z}$-grading on $E=E^{\text{ev}}\oplus E^{\text{odd}}$ where $E^{\text{ev}}=\bigoplus_{i} E_{2i}$ and $E^{\text{odd}}=\bigoplus_{i} E_{2i+1}$. If $T: C^\infty(V,E) \to C^\infty(V,E)$ is an even smoothing operator, meaning that $T$ preserves the decomposition $C^\infty(V,E)=C^\infty(V,E^{\text{ev}})\oplus C^\infty(V,E^{\text{odd}})$ and  its Schwartz kernel belongs to $C^\infty(V\times V, E\boxtimes E^\ast)$, its supertrace $\operatorname{Str}(T)$ is defined as
	\begin{equation}
	\operatorname{Str}(T) = \operatorname{Tr}\Big(T: C^\infty(V,E^{\text{ev}}) \to C^\infty(V,E^{\text{ev}})\Big)-\operatorname{Tr}\Big(T: C^\infty(V,E^{\text{odd}})\to C^\infty(V,E^{\text{odd}})\Big)
	\end{equation}
	where $\operatorname{Tr}$ is the usual trace on smoothing operators(trace class).
		
	Let $\varphi:V \to V$ be a smooth map with only simple fixed points (meaning that $\left|\det(\varphi_{\ast,m}-1)\right|$ is invertible for all its fixed point $m\in V$). Assume in addition that there is a bundle map $\zeta: \varphi^\ast E\to E$ which induces the smooth map $\zeta: C^\infty(V,\varphi^\ast E)\to C^\infty(V,E)$. Consider the composition
	\[
	T: C^\infty(V,E) \xrightarrow{\varphi^\ast} C^\infty(V,\varphi^\ast E)\xrightarrow{\zeta} C^\infty(V,E).
	\]
	If $T=\zeta\circ \varphi^\ast$ commutes with the differential $d$ defined in \eqref{eq-sum-diff}, it lifts to an endomorphism on the cohomology ring
	\[
	T: H^\ast(E,d) \to H^\ast(E,d).
	\]
	The Lefschetz fixed point formula is about to calculate the alternating sum
	\begin{equation}\label{Lefschetz-number}
	\sum_{i=1}^m (-1)^{i}\operatorname{Tr}\left(T: H^i(E,d) \to H^i(E,d)\right).
	\end{equation}
	
	The Hodge decomposition theorem(see \cite[Chapter~1]{Gilkey84}) identifies the cohomology ring $H^\ast (E,d)$ with the kernel of Laplacian $\Delta=(d+d^\ast)^2$ as vector spaces. Under this light, the alternating sum \eqref{Lefschetz-number} can also be expressed as a supertrace
	\begin{equation}\label{Lefschetz-supertrace}
	\operatorname{Str}\left(Te^{-t\Delta}: C^\infty(V,E)\to C^\infty(V,E)\right).
	\end{equation}
	Thanks to the fact that $T$ commutes with $d$, they can be simultaneously diagonalized. Moreover, since each nonzero eigenspace of $\Delta$ in $C^\infty(V,E^{\text{ev}})$ can be identified with that in $C^\infty(V,E^{\text{odd}})$ by $d+d^\ast$, only zero eigenvalue contributes to \eqref{Lefschetz-supertrace}. The supertrace \eqref{Lefschetz-supertrace} is therefore independent of $t$. 
	
	The role of deformation spaces and groupoids is to help us justify and compute  the limit of \eqref{Lefschetz-supertrace} as $t\to 0$. To make it more precise, we shall build a groupoid  whose space of smooth functions is equipped with a continuous family of functionals $\operatorname{Str}_t$ parametrized by $t\in [0,1]$ and whose corresponding set of pseudodifferential operators contains $Te^{-t\Delta}$. We can recover \eqref{Lefschetz-supertrace} by applying $\operatorname{Str}_t$ to the integral kernel of $Te^{-t\Delta}$ and compute its value by setting $t=0$.
	
	\begin{theorem}{(The Lefschetz fixed point formula)}\label{thm-lef-fix}
	\[
	\operatorname{Str}\left(Te^{-t\Delta}: C^\infty(V,E)\to C^\infty(V,E)\right)=\sum_{m\in M}\left|\det(\varphi_{\ast,m}-1)\right|^{-1}\operatorname{str}(\zeta_m).
	\]
	\end{theorem}
	
	The correspondence between abstract groupoids and pseudodifferential calculus shall be reviewd in Section~\ref{groupoids-pseudo}. The groupoid that serves our purpose will be bulit in Section~\ref{deformation-space} and Section~\ref{coefficient-bundle}. The final calculation of \eqref{Lefschetz-supertrace} will be performed in Section~\ref{final-calculation}.


	\section{Groupoids and pseudodifferential operators}\label{groupoids-pseudo}
	In this section, we shall quickly review basics of groupoids and the correspondence between Lie groupoids and pseudodifferential operators established in \cite{NistorWeinsteinXu99}. For our purpose, we shall mainly focus on smoothing operators.
	
	A Lie groupoid, usually denoted as $G\rightrightarrows G^{(0)}$, consists of the following data:
	\begin{enumerate}
	\item Two smooth manifolds $G=G^{(1)}$ and $G^{(0)}$ with two submersions $s,r: G\to G^{(0)}$ called source and range maps of the groupoid.
	
	Roughly speaking, $G^{(0)}$ can be thought of as a set of points and $G^{(1)}$ as a set of arrows between those points. The source and target maps send an arrow to the initial and terminal points respectively;
	\item An associative multiplication map $m: G^{(2)} \to G$ where $G^{(2)}=\{(\gamma,\eta)\in G\times G \mid s(\gamma)=r(\eta)\}$. 
	
	Given two arrows, if the initial point of the first arrow coincides with the terminal point of the second one, they can be "connected" to produce a new arrow. This type of "connection" determines the multiplication map $m$; We shall write $\gamma \cdot \eta$ for $m(\gamma, \eta)$.
	
	\item A smooth map $\varepsilon: G^{(0)}\to G^{(1)}$.
	
	Given a point $x\in G^{(0)}$, it can be viewed as an arrow(loop) which goes from $x$ to $x$;
	
	\item A smooth map $\iota:G\to G$ such that
			\[
			\iota(\gamma)\circ \gamma = \varepsilon(s(\gamma)), \quad \gamma\circ\iota(\gamma)=\varepsilon(r(\gamma)) \quad \forall \gamma\in G.
			\]
			
		Given an arrow, the map $\iota$ reverse its direction. We shall write $\gamma^{-1}$ for $\iota(\gamma)$.
	\end{enumerate}
	
	Let $G\rightrightarrows G^{(0)}$ be a Lie groupoid, $E$ be a vector bundle over the unit space $G^{(0)}$. Let $G_x=s^{-1}(x)$ and $G^y=r^{-1}(y)$ be the corresponding source and range fibers where $x,y\in G^{(0)}$. Define $\mathbb{E}$ to be the tensor product vector bundle $r^\ast E \otimes s^\ast E^\ast \to G$. 
	
	\begin{definition}
	$\left(P_x,x\in G^{(0)}\right)$ is said to be \emph{a family of pseudodifferential operators} on $\mathbb{E}\to G$ if $P_x$ is a pseudodifferential operators acting on $C^\infty_c(G_x,r^\ast E)$ for all $x\in G^{(0)}$.
	\end{definition}
	
	Let $g\in G$, it defines a right translation $U_g: C^\infty(G_{s(g)},r^\ast E)\to C^\infty(G_{r(g)},r^\ast E)$ which is given by $(U_gf)(h)=f(hg)$.
	
	\begin{definition}\label{def-equi-family}
	$\left(P_x,x\in G^{(0)}\right)$ is called an \emph{equivariant} family of pseudodifferential operators if
	\begin{equation}\label{eq-equivariance}
	U_gP_{s(g)} = P_{r(g)}U_g
	\end{equation}
	for all $g\in G$.
	\end{definition}
	
	In the following discussion, let us assume that $\left(P_x,x\in G^{(0)}\right)$ consists of smoothing operators. That is, the Schwartz kernel $k_x$ of $P_x$ belongs to $C^\infty(G_x\times G_x,r^\ast E \boxtimes r^\ast E^\ast)$ for all $x\in G^{(0)}$. The importance of the equivariance is that it allows us to reduce the Schwartz kernel. Indeed, in terms of the Schwartz kernel, equation \eqref{eq-equivariance} read
	\[
	k_{r(g)}(h^\prime, h) = k_{s(g)}(h^\prime g, hg)
	\]
	which is then equivalent to 
	\[
	k_{s(g)}(g^\prime, g) = k_{r(g)}(g^\prime g^{-1},r(g)).
	\]
	
	\begin{definition}
	Let $\left(P_x,x\in G^{(0)}\right)$ be an equivariant family of smoothing operators. Its \emph{reduced kernel} $k_P$, which is a (not necessarily smooth) section of $\mathbb{E}\to G$, is defined to be
	\[
	k_P(g) = k_{s(g)}(g,s(g))
	\]
	where $k_x$ is the Schwartz kernel of $P_x$.
	\end{definition}
	According to the definition, $k_P(g) = k_{s(g)}(g,s(g))\in E_{r(g)}\otimes E^\ast_{s(g)}$. So, $k_p$ is indeed a section of $\mathbb{E}\to G$.
	
	\begin{definition}\label{def-smooth-family}
	The family $\left(P_x,x\in G^{(0)}\right)$ is said to be a \emph{smooth} if its reduced kernel $k_p$ is a smooth section.
	\end{definition}
	
	\begin{remark}
	In \cite{NistorWeinsteinXu99}, in order to discuss composition of pseudodifferential operators, the authors need some conditions on support. Since in this paper the composition formula is not necessary, we are able to get rid of the support conditions and consider larger class of pseudodifferential operators. 	\end{remark}

	\begin{example}\label{ex-pair-groupoid}
	Let $V$ be a smooth manifold and $G^{(1)}=V\times V\rightrightarrows V=G^{(0)}$ be the pair groupoid where the range and source maps are the projections onto the first and second variables respectively. The source fiber over $v\in V$ is $G_v=V\times \{v\}$. A family of pseudodifferential operators $\left(P_v,v\in V\right)$ on $G$ is simply a family of pseudodifferential operators $\{P_v\}_{v\in V}$ on $V$ parametrized by $v\in V$. If $\left(P_v,v\in V\right)$ is equivariant in the sense of Definition~\ref{def-equi-family}, then \eqref{eq-equivariance} read
	\[
	P_v U_{(v,v^\prime)} = U_{(v,v^\prime)}P_{v^\prime}.
	\]
	Notice that $U_{(v,v^\prime)}$ identifies $G_v$ with $G_{v^\prime}$. The above equation simply means $P_v=P_{v^\prime}$. In a word, equivariant family in this context is constant family. If $P_v$ is a smoothing operator, the reduced kernel $k_P$ is precisely the kernel of $P_v$ which belongs to $C^\infty(V\times V)$.
	\end{example}

	\section{Deformation space}\label{deformation-space}
	In this section, we briefly review the deformation to the normal cone construction. For a more detailed account see \cite{Mohsen19Witten,Debord17blowup}.
	
	To the embedding $M\subseteq V$, the associated deformation to the normal cone $\mathbb{N}_V M$ is a smooth manifold whose underlying set is 
	\[
	\mathbb{N}_V M= N_V M \sqcup V\times (0,1].
	\]
	Let $V\supseteq U \xrightarrow{\varphi} \mathbb{R}^n=\mathbb{R}^p\times \mathbb{R}^q$ be a local coordinate chart of $V$ such that $M\cap U = \varphi^{-1}\left(\varphi(U) \cap \mathbb{R}^p\times \{0\}\right)$. Then $\mathbb{N}_V M \supseteq \mathbb{N}_U (M\cap U) \xrightarrow{\phi} \mathbb{R}^{n+1}$ is seen as a local coordinate chart of the deformation space with the homeomorphism $\phi$ given by
	\begin{equation}\label{smooth-structure}
	\begin{split}
	(v,t) &\mapsto (\varphi_p(v),\frac{1}{t}\varphi_q(v),t);\\
	(X,m) &\mapsto (\varphi_{p,\ast}(X),\varphi_q(m),0),
	\end{split}
	\end{equation}
	where $\varphi_p$ and $\varphi_q$ are the first $p$ and last $q$ components of $\varphi$ respectively.
	
	Equivalently, the smooth structure is determined by declaring the following functions to be smooth:
	\begin{enumerate}
		\item If $f$ is a smooth functions on $V$ then the assignment $\mathbb{N}_V M \to \mathbb{R}$
		\begin{align*}
		(v,t) &\mapsto f(v) \\
		X_m &\mapsto f(m)
		\end{align*}
		where $X_m$ is a normal vector at $m\in M$, is a smooth function on $\mathbb{N}_V M$.
		\item If $f$ is a smooth functions on $V$ that vanishes to order $r$ on $M$, then the assignment $\mathbb{N}_V M \to \mathbb{R}$
		\begin{align*}
		(v,t) &\mapsto \frac{1}{t^r}f(v) \\
		X_m &\mapsto \frac{1}{r!}X_m^r(f)
		\end{align*}
		is a smooth function on $\mathbb{N}_V M$.
	\end{enumerate}
	
	The deformation to the normal cone construction is functorial in the following sense. Given a commutative diagram
	\begin{equation*}
	\xymatrix{
	V^\prime \ar[r]  & V \\
	M^\prime \ar[r]  \ar[u] &  M\ar[u]
	}
	\end{equation*}
	where the vertical maps are inclusion of submanifolds and the horizontal maps are any smooth maps, there is an induced map between deformation spaces $\mathbb{N}_V^\prime M^\prime \to \mathbb{N}_V M$. Moreover, if the horizontal maps in the above diagram are submersions, then the induces map between deformation spaces is also a submersion.

	The following result will be important to us.
	\begin{proposition}\label{Integral-deformation-space}
		Fix a smooth measure $\mu_V$ on $V$, it induces smooth measures $\mu_m$ on $T_m V/T_m M$ for each $m\in M$ and a smooth measure $\mu_M$ on $M$. Let $f\in C^\infty(\mathbb{N}_V M)$ and assume that
		\[
		t^{\operatorname{dim} M-\operatorname{dim} V}\int_V f(v,t)d\mu_V(v)< \infty
		\]
		is uniformly bounded with respect to $t\in (0,1]$. Then it converges, as $t\to 0$, to 
		\[
		\int_M \int_{T_m V/T_m M} f(X_m)d\mu_m(X_m)d\mu_M(m)<\infty.
		\]
	\end{proposition}
	
	\begin{proof}
	Let $U$ be an open subset of $V$. According to the smooth structure \eqref{smooth-structure} of the deformation space, $\mathbb{N}_U (M\cap U)$ is a local coordinate chart of $\mathbb{N}_V M$. Without loss of generality, we may assume that $f$ is supported inside the coordinate chart $\mathbb{N}_U (M\cap U)$. Then the result follows from a local calculation:
	\begin{align*}
	t^{\operatorname{dim} M-\operatorname{dim} V}\int_V f(v,t) d\mu_V(v) &= t^{-q}\int_{\mathbb{R}^n} f(v_1,\cdots,v_{p},tv_{p+1},\cdots,tv_{p+q},t) d\mu(v) \\
	&=\int_{\mathbb{R}^n} f(v_1,\cdots,v_{p+q},t) d\mu(v) \\
	&\to \int_{\mathbb{R}^n} f(v_1,\cdots,v_{p+q},0) d\mu(v).
	\end{align*}
	\end{proof}

	Consider the diagonal embedding $M\hookrightarrow V\times V$ and its associated deformation to the normal cone $\mathbb{N}_{V^2} M$. Thanks to the functoriality of the deformation construction, $\mathbb{N}_{V^2} M$ is a Lie groupoid with the unit space being $\mathbb{N}_V M$:
	\[
	\mathbb{N}_{V^2} M \rightrightarrows \mathbb{N}_VM.
	\]
The range and source maps are induced from the the commutative diagrams
	\begin{align*}
	\xymatrix{
	V\times V \ar[r]^{\pi_1} & V\\
	M\ar[r] \ar[u] & M \ar[u]
	}
	&
	\,
	&
	\xymatrix{
	V\times V \ar[r]^{\pi_2} & V\\
	M\ar[r] \ar[u] & M \ar[u]
	}
	\end{align*}
	where in both diagrams the left vertical maps are the diagonal embeddings, the right vertical maps are the embeddings of $M$ into $V$, the lower horizontal maps are the identity map and $\pi_1, \pi_2$ are the projections onto the first and second factors respectively. 
	
	To emphasis the groupoid structure and hint the similarity to the tangent groupoid, we shall write $\mathbb{T}_M V$ for $\mathbb{N}_{V^2} M$ and call it \emph{the relative tangent groupoid}. 
	
	\begin{definition}\label{def-traces}
		Let $k$ be a smooth function on the relative tangent groupoid $\mathbb{T}_M V$. We shall use $\operatorname{Tr}_t(k)$ to denote the following quantity
		\[
		\operatorname{Tr}_t(k)=\int_V k(v,v,t) d\mu_V(v).
		\]
		for all $t\neq 0$. And define
		\[
		\operatorname{Tr}_0(k)=\int_M\int_{T_m V/T_m M} k(X_m,X_m)d\mu_m(X_m) d\mu_M(m).
		\]
	\end{definition}
	
	
	Let $\varphi: V\to V$ be a smooth map with only simple fixed points and let $M$ be the set of its fixed points. It is straightforward to see that 
	\[
	\mathbb{T}_M V = \sqcup_{m\in M} T_m V\oplus T_m V \sqcup V\times V\times (0,1].
	\]
	The source map sends $(v_1,v_2,t)$ to $(v_2,t)$ when $t\neq 0$ and sends $(X_m,Y_m)\in T_mV\oplus T_m V$ to $Y_m$. The range map sends $(v_1,v_2,t)$ to $(v_1,t)$ when $t\neq 0$ and sends $(X_m,Y_m)\in T_mV\oplus T_m V$ to $X_m$.
	
	The following proposition is a direct consequence of Proposition~\ref{Integral-deformation-space}.
	\begin{proposition}\label{prop-cont-traces}
	Let $n=\operatorname{dim} V$, if $k\in C^\infty(\mathbb{T}_MV)$ such that $$f(t)=t^{-n}\operatorname{Tr}_t(k)$$ is uniformly bounded with respect to $t\in (0,1]$. Then $f(t)$ converges to $\operatorname{Tr}_0(k)$ as $t\to 0$.
	\end{proposition}

	If $k\in C^\infty(\mathbb{T}_M V)$, then $k_\varphi(x,y)=k(\varphi(x),y)$ is still a smooth function on the relative tangent groupoid $\mathbb{T}_M V$. Its valued at $(X,Y)\in T_mV\oplus T_mV$ can be computed
	\[
	k_\varphi(X,Y) = k(\varphi_\ast(X),Y)
	\]
	where $\varphi_{\ast,m}: T_mV \to T_mV$ is the differential of $\varphi$ at $m\in M$. Assume, in addition, that $k$ satisfies 
	\begin{equation}
	k(X_m,Y_m)=k(X_m-Y_m,0)
	\end{equation}
	which is saying that the pseudodifferential operator, of which $k$ is the Schwartz kernel, is translation invariant. Then
	\begin{equation}
	\begin{split}
	\operatorname{Tr}_0(k_\varphi)&=\sum_{m\in M} \int_{T_m V} k((\varphi_\ast-1)X_m,0)d\mu_m(X_m) \\
	&= \sum_{m\in M}\int_{T_m V} k(X_m,0)\left|\det(\varphi_{\ast,m}-1)\right|^{-1}d\mu_m(X_m)\\
	&= \sum_{m\in M}\left|\det(\varphi_{\ast,m}-1)\right|^{-1}\int_{T_m V}k(X_m,0)d\mu_m(X_m).
	\end{split}
	\end{equation}

	\section{The coefficient bundle}\label{coefficient-bundle}
	In this section, we shall build a vector bundle over the relative tangent groupoid $\mathbb{T}_M V$ to account for vector bundles appeared in the elliptic complex \eqref{elliptic-complex}.
	
	For the deformation space $\mathbb{N}_V M$, there is a canonical map $\mathbb{N}_V M \to V$ which sends $(v,\lambda)\in V\times (0,1]$ to $v$ and $X_m\in T_mV/T_mM$ to $m\in M\subset V$. Let $F$ be the pullback of the vector bundle $E$ defined in \eqref{eq-vector-bundle-E} along this canonical map and let $\mathbb{E}$ be the tensor product $r^\ast F \otimes s^\ast F^\ast \to \mathbb{T}_M V$. It is easy to check that
	\[
	\mathbb{E}_{(v_1,v_2,t)} \cong E_{v_1}\otimes E^\ast_{v_2}, \qquad \text{and}\qquad \mathbb{E}_{(X_m,Y_m)}\cong \End(E_m)
	\]
where $(v_1,v_2,t)\in V\times V\times (0,1]$ and $(X_m,Y_m)\in T_mV \times T_m V$.

	The following Definition and Proposition are supertrace version of Definition~\ref{def-traces} and Proposition~\ref{prop-cont-traces}:

	\begin{definition}
		If $k$ is a smooth section of $\mathbb{E} \to \mathbb{T}_M V$. We shall use $\operatorname{Str}_t(k)$ to denote the following quantity
		\[
		\operatorname{Str}_t(k)=\int_V \operatorname{str}\left(k(v,v,t)\right) d\mu_V(v).
		\]
		for all $t\neq 0$. And define
		\[
		\operatorname{Str}_0(k)=\sum_{m\in M}\int_{T_m V}  \operatorname{str}\left(k(X_m,X_m)\right)d\mu_m(X_m)
		\]
		where $ \operatorname{str}: \End(E_v)\to \mathbb{C}$ is the supertrace which reflects the $\mathbb{Z}/2\mathbb{Z}$-grading of $E_v$.
	\end{definition}

	\begin{proposition}\label{prop-cont-supertraces}
	Let $n=\operatorname{dim} V$, if $k\in C^\infty(\mathbb{T}_MV)$ such that $$f(t)=t^{-n}\operatorname{Str}_t(k)$$ is uniformly bounded with respect to $t\in (0,1]$. Then $f(t)$ converges to $\operatorname{Str}_0(k)$ as $t\to 0$.
	\end{proposition}

	Now, if $\Delta=(d+d^\ast)^2: C^\infty(V,E) \to  C^\infty(V,E)$ is an elliptic differential operator of order $2s$ with principal symbol 
	\begin{equation}\label{principal-symbol}
	\sigma(\Delta)(v,\xi) = \sum_{|\alpha|=2s}a_\alpha(v) \xi^\alpha
	\end{equation}
	where $a_\alpha(v)$ are positive definite matrices in $\End(E_v)$.
	It is well known that the heat kernel $e^{-t^{2s}\Delta}$ is a smoothing operator acting on $C^\infty_c(V,E)$. Let $X$ be the coordinates on the tangent space $T_vV$, then the heat kernel 
	\[
	\exp\left({-\sum_{|\alpha|=2s}a_\alpha(v) \frac{\partial^\alpha}{\partial X^\alpha}}\right)
	\]
	is a smoothing operator acting on $C^\infty_c(T_vV, E_v)$. Let $\left(Q_x,x\in G^{(0)}\right)$ be the family of smoothing operators on the relative tangent groupoid $\mathbb{T}_M V$ which is defined by
	\begin{equation}\label{eq-family-heat-kernel}
	Q_x=
	\begin{cases}
	\exp\left({-t^{2s}\Delta}\right) & x=(v,t)\in V\times (0,1] \\
	\exp\left({-\sum_{|\alpha|=2s}a_\alpha(m) \frac{\partial^\alpha}{\partial X^\alpha}}\right) & x=X_m\in T_m V.
	\end{cases}
	\end{equation}
	
	\begin{proposition}\label{prop-smooth-family-diff}
	The family $\left(Q_x,x\in G^{(0)}\right)$ is equivariant and smooth in the sense of definition~\ref{def-equi-family} and \ref{def-smooth-family}.
	\end{proposition}
	
	\begin{proof}
	The equivariance can be shown as in Example~\ref{ex-pair-groupoid}.
	
	The smoothness can be checked locally. Indeed, pick a local coordinate chart $V\supset U\xrightarrow{\varphi} \mathbb{R}^n$ such that $M\cap U$ contains only single point $m\in M$. Then $\mathbb{T}_{M\cap U} U\subset \mathbb{T}_M V$ forms a local coordinate chart whose diffeomorphism into Euclidean space is given by
	\begin{equation}\label{eq-rel-tangent-local-coordinate}
	\mathbb{T}_{M\cap U} U \to \mathbb{R}^{2n+1}:
	\begin{cases}
	(u_1,u_2,t) \mapsto \left(\frac{\varphi(u_1)-\varphi(m)}{t}, \frac{\varphi(u_2)-\varphi(m)}{t}, t\right);\\
	(X_m, Y_m,0) \mapsto \left(\varphi_\ast(X_m),\varphi_\ast(Y_m),0\right).
	\end{cases}
	\end{equation}
	
	By the definition of principal symbol \eqref{principal-symbol}, $\Delta$ has the following local expression on $U$:
	\[
	\Delta= -\sum_{|\alpha|=2s}a_\alpha(u) \frac{\partial^\alpha}{\partial u^\alpha}+ \text{lower order terms}.
	\]
	Under the local coordinate chart \eqref{eq-rel-tangent-local-coordinate}, taken as an operator on $\mathbb{T}_M V_{(v,t)}$, $t^{2s}\Delta$ has following local expression:
	\[
	t^{2s}\Delta= -\sum_{|\alpha|=2s}a_\alpha(tu+m) \frac{\partial^\alpha}{\partial u^\alpha}+ t\cdot \text{lower order terms}.
	\]
	The heat kernel of 
	\[
	\exp\left( -\sum_{|\alpha|=2s}a_\alpha(tu+m) \frac{\partial^\alpha}{\partial u^\alpha}+ t\cdot \text{lower order terms}\right)
	\]
	is smooth in all variables, in particular in $t\in [0,1]$. And its value at $t=0$ is the heat kernel of 
	\[
	\exp\left(-\sum_{|\alpha|=2s}a_\alpha(m) \frac{\partial^\alpha}{\partial X^\alpha}\right).
	\]
	\end{proof}

	\section{Lefschetz fixed point formula}\label{final-calculation}
	In this section we apply the gadget developed above to calculate the supertrace \eqref{Lefschetz-supertrace} as $t\to 0$.
	
	Let $k_P\in C^\infty(\mathbb{T}_M V, \mathbb{E})$ be the reduced kernel of $\left(Q_x,x\in G^{(0)}\right)$ defined in \eqref{eq-family-heat-kernel}. Then we have
	\begin{multline}\label{lines-supertraces}
	\operatorname{Str}\left(Te^{-t\Delta}: C^\infty(V,E)\to C^\infty(V,E)\right)
	= \int_V \operatorname{str}\left(\zeta \cdot k_P(\varphi(v),v,t)\right)t^{-n}d\mu_V(v) \\ = t^{-n}\operatorname{Str}_t\Big((\zeta\cdot k_p)_\varphi\Big) 
	\end{multline}
	where $(\zeta\cdot k_p)_\varphi$ is a smooth section of the bundle $\mathbb{E}\to \mathbb{T}_M V$ such that 
	\[
	\begin{cases}
	(\zeta\cdot k_p)_\varphi(v_1,v_2,t) = \zeta \cdot k_P(\varphi(v_1),v_2,t)\\
	(\zeta\cdot k_p)_\varphi(X_m,Y_m) = \zeta_m \cdot k_p(\varphi_\ast X_m,Y_m).
	\end{cases}
	\]
	By Proposition~\ref{prop-cont-supertraces}, \eqref{lines-supertraces} converges to 
	\begin{align*}
	\operatorname{Str}_0 \Big( (\zeta\cdot k_p)_\varphi \Big)&= \sum_{m\in M}\int_{T_m V}  \operatorname{str}\left(\zeta_m\cdot k(\varphi_\ast X_m,X_m)\right)d\mu_m \\
	&=\sum_{m\in M}\left|\det(\varphi_{\ast,m}-1)\right|^{-1}\operatorname{str}(\zeta_m\cdot \int_{T_mV}k(X_m,0)d\mu_m).\\
	\end{align*}
	By Proposition~\ref{prop-smooth-family-diff}, $k(X_m,0)$ is the heat kernel of $-\sum_{|\alpha|=2s}a_\alpha(x) \frac{\partial^\alpha}{\partial X^\alpha}$. 
	
	\begin{lemma}\label{lem-integral-heat-kernel}
	Let $\alpha=(\alpha_1,\alpha_2,\dots,\alpha_n)$ be a multi-index such that $|\alpha|=2s$ is an even integer. Let
	\[
	K_0=-A \frac{\partial^\alpha}{\partial X^\alpha}
	\]
	be a differential operator on $\mathbb{R}^n$ with coefficient $A$ being positive definite matrix. Let $k_0$ be the heat kernel of $K_0$, then
	\[
	\int_{\mathbb{R}^n} k_0(X,0)dX = 1.
	\]
	\end{lemma}
	
	\begin{proof}
	Fourier transformation turn the heat equation
	\begin{equation}\label{eq-heat-eq}
	\frac{\partial u}{\partial t} +A \frac{\partial^\alpha u}{\partial X^\alpha} = 0 
	\end{equation}
	into the ordinary differential equation
	\[
	\mathscr{F}(u_t)+A\xi^\alpha \mathscr{F}(u)=0
	\]
	where $\mathscr{F}$ is the Fourier transformation. Therefore the fundamental solution of \eqref{eq-heat-eq} is
	\[
	\mathscr{F}^{-1}\Big(\exp(-A\xi^\alpha)\Big) 
	\]
	where $\mathscr{F}^{-1}$ is the Fourier inverse transformation.
	Then
	\[
	\int_{\mathbb{R}^n} k_0(X,0)dX = \int e^{-A\xi^\alpha} e^{ix\cdot \xi}d\xi dx = 1.
	\]
	The final equation follows from the simple fact that
	\[
	\int_{\mathbb{R}^n} \mathscr{F}^{-1}(u)(x) dx = u(0).
	\]
	\end{proof}

	By Lemma~\ref{lem-integral-heat-kernel}, we have
	\[
	\int_{T_mV}k(X_m,0)d\mu_m=1.
	\]
	Therefore 
	\[
	\operatorname{Str}_0 \Big( (\zeta\cdot k_p)_\varphi \Big) = \sum_{m\in M}\left|\det(\varphi_{\ast,m}-1)\right|^{-1}\operatorname{str}(\zeta_m).
	\]
	This completes the proof of Theorem~\ref{thm-lef-fix}.
	
\section*{Acknowlegement}
The author would like to thank Professor Weiping Zhang for helpful suggestions on an early draft.

	\bibliography{Refs} 

\def\cprime{$'$} \def\cprime{$'$}
\begin{thebibliography}{NWX99}

\bibitem[AB66]{AtiyahBott66}
M.~F. Atiyah and R.~Bott.
\newblock A {L}efschetz fixed point formula for elliptic differential
  operators.
\newblock {\em Bull. Amer. Math. Soc.}, 72:245--250, 1966.

\bibitem[Bis11]{Bismut11}
Jean-Michel Bismut.
\newblock {\em Hypoelliptic {L}aplacian and orbital integrals}, volume 177 of
  {\em Annals of Mathematics Studies}.
\newblock Princeton University Press, Princeton, NJ, 2011.

\bibitem[Con94]{Connes94}
Alain Connes.
\newblock {\em Noncommutative geometry}.
\newblock Academic Press, Inc., San Diego, CA, 1994.

\bibitem[DS17]{Debord17blowup}
Claire Debord and Georges Skandalis.
\newblock Blowup constructions for lie groupoids and a boutet de monvel type
  calculus, 2017.

\bibitem[Gil84]{Gilkey84}
Peter~B. Gilkey.
\newblock {\em Invariance theory, the heat equation, and the {A}tiyah-{S}inger
  index theorem}, volume~11 of {\em Mathematics Lecture Series}.
\newblock Publish or Perish, Inc., Wilmington, DE, 1984.

\bibitem[Hig10]{Higson10}
Nigel Higson.
\newblock The tangent groupoid and the index theorem.
\newblock In {\em Quanta of {Maths}}, volume~11 of {\em Clay Math. Proc.},
  pages 241--256. Amer. Math. Soc., Providence, RI, 2010.

\bibitem[HY19]{HigsonYi19}
Nigel Higson and Zelin Yi.
\newblock Spinors and the tangent groupoid.
\newblock {\em Doc. Math.}, 24:1677--1720, 2019.

\bibitem[Moh19]{Mohsen19Witten}
Omar Mohsen.
\newblock Witten deformation using lie groupoids, 2019.

\bibitem[NWX99]{NistorWeinsteinXu99}
Victor Nistor, Alan Weinstein, and Ping Xu.
\newblock Pseudodifferential operators on differential groupoids.
\newblock {\em Pacific J. Math.}, 189(1):117--152, 1999.

\end{thebibliography}
	\bibliographystyle{alpha}
	
	\noindent {\small  Chern Institute of Mathematics, Nankai University, Tianjin, P. R. China 300071.}
	
	
	\noindent{\small Email: zelin@nankai.edu.cn}

\end{document}